\documentclass{notices}

\usepackage{etex}
\usepackage{amsmath,amssymb,amsthm,bbm,mathtools,comment}
\usepackage[shortlabels]{enumitem}
\usepackage[mathscr]{euscript}
\usepackage{tikz,babel,adjustbox}
\usepackage{cmtiup}
\usepackage{amsfonts}
\usepackage{graphicx}
\usepackage{caption}
\usepackage{verbatim}
\usepackage{array}
\usepackage[frame,cmtip,arrow,matrix,line,graph,curve]{xy}
\usepackage{graphpap, color, pstricks}
\usepackage{pifont}
\usepackage{cmtiup}

\usepackage{multirow}

\usepackage[margin=.5in, top=.7in, bottom=1in]{geometry}

\title{
Threshold phenomena for random discrete structures
}

\author{
  Jinyoung Park
  \affil{
    The author will be an assistant professor of mathematics at Courant Institute of Mathematical Sciences, NYU, starting in Fall 2023. Her email address is jinyoungpark@nyu.edu. Her research is supported by NSF grant DMS-2153844.
    }
}

\theoremstyle{plain}
\newtheorem{theorem}{Theorem}[section]

\newtheorem{prop}[theorem]{Proposition}

\newtheorem{conjecture}[theorem]{Conjecture}

\newtheorem{mydef}[theorem]{Definition}

\newtheorem{question}[theorem]{Question}

\theoremstyle{remark}

\newtheorem{example}[theorem]{Example}
\newtheorem{remark}[theorem]{Remark}

\newcommand{\pr}{\mathbb{P}}

\newcommand{\sub}[0]{\subseteq}

\newcommand{\beq}[1]{\begin{equation}\label{#1}}
\newcommand{\enq}[0]{\end{equation}}

\newcommand{\pE}[0]{p_{\mathbb E}}

\newcommand{\ra}[0]{\rightarrow}

\newcommand{\eps}[0]{\varepsilon }

\newcommand{\nin}[0]{\noindent}

\newcommand{\cF}{\mathcal{F} }
\newcommand{\cG}{\mathcal{G} }

\begin{document}

\maketitle

\section{Erd\H{o}s-R\'enyi model}\label{sec.intro}

To begin, we briefly introduce a model of random graphs. Recall that a graph is a mathematical structure that consists of vertices (nodes) and edges. 

\begin{figure}[h!]\centering\includegraphics[height=1in]{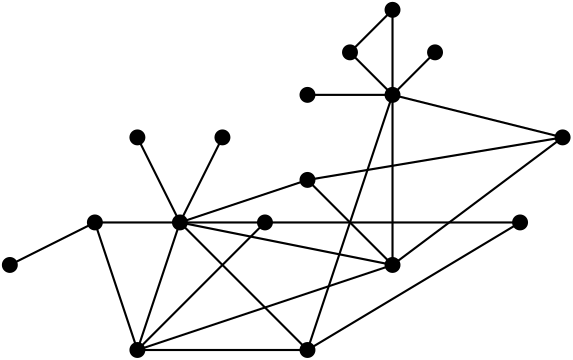}\caption{A graph}
\end{figure}

Roughly speaking, a random graph in this article means that, given a vertex set, the existence of each potential edge is decided at random. We will specifically focus on the \textit{Erd\H{o}s-R\'enyi random graph} (denoted by $G_{n,p}$), which is defined as follows.

Consider $n$ vertices that are labelled from 1 to $n$. 

\begin{figure}[h!]\centering\includegraphics[height=1.3in]{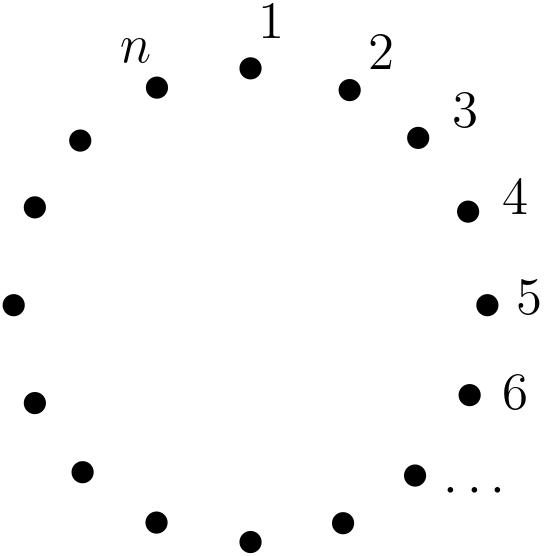}\end{figure}

\nin Observe that on those $n$ vertices, there are potentially ${n \choose 2}$ edges, that is, the edges labelled $\{1,2\},\{1,3\}, \ldots, \{n-1, n\}$. Given a probability $p \in [0,1]$, include each of the $n \choose 2$ potential edges with probability $p$, where the choice of each edge is made independently from the choices of the other edges.

\begin{example}\label{ex1} As a toy example of the Erd\H{o}s-R\'enyi random graph, let's think about what $G_{n,p}$ looks like when $n=3$ and the value of $p$ varies. First, if $p=1/2$, then $G_{n,p}$ has the probability distribution as in Figure \ref{fig:Gnp_unif} defined on the collection of eight graphs. Observe that each graph is equally likely (since each potential edge is present with probability $1/2$ independently).

\begin{figure}[h!]\centering \includegraphics[height=1.8in]{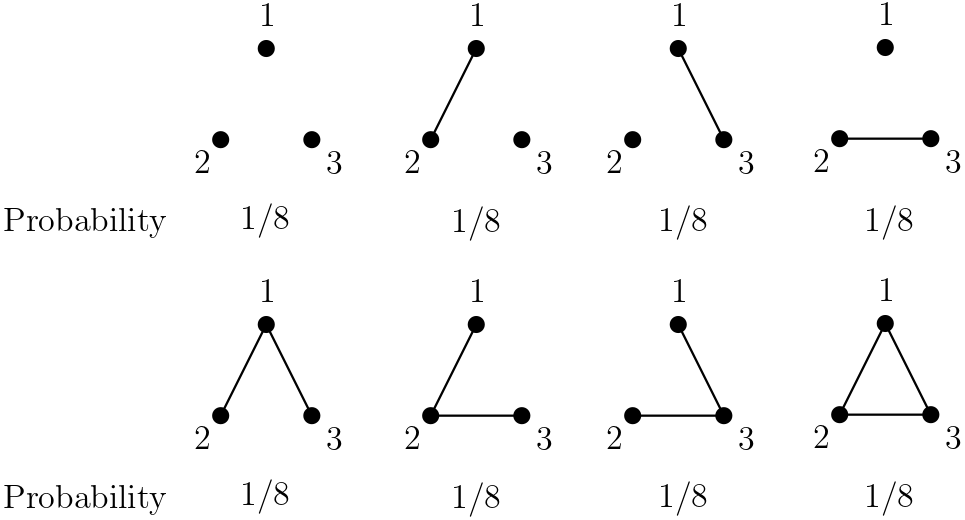}\caption{$G_{3,1/2}$}\label{fig:Gnp_unif}\end{figure}

Of course, we will have a different probability distribution if we change the value of $p$. For example, if $p$ is closer to 0, say 0.01, then $G_{n,p}$ has the distribution as in Figure \ref{fig:Gnp_sparse}, where sparser graphs are more likely (as expected). On the other hand, if $p$ is closer to 1, then denser graphs will be more likely.

\begin{figure}[h!]\centering \includegraphics[height=1.8in]{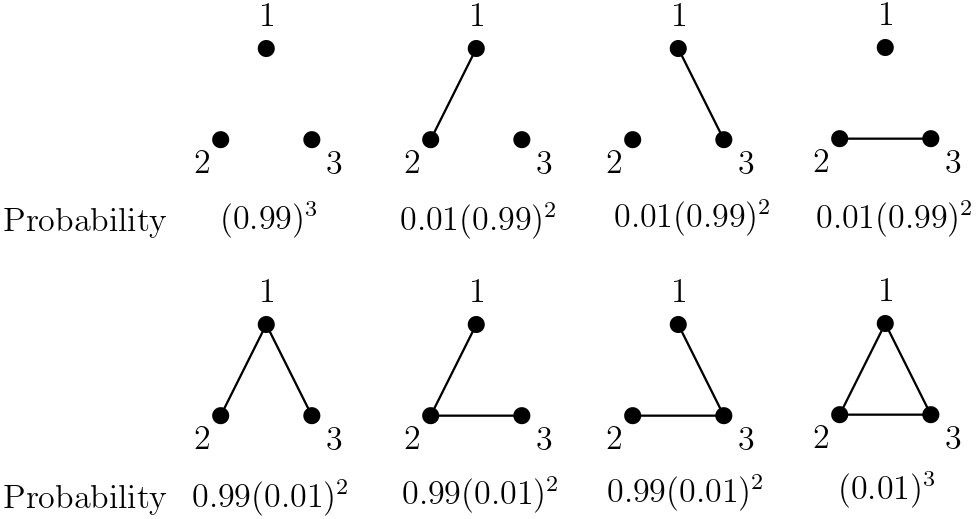}\caption{$G_{3,0.01}$}\label{fig:Gnp_sparse}\end{figure}

\end{example}

In reality, when we consider $G_{n,p}$, $n$ is a large (yet finite) number that tends to infinity, and $p=p(n)$ is usually a function of $n$ that tends to zero as $n \ra \infty$. For example, $p=1/n$, $p=\log n /n$, etc.

As we saw in Example \ref{ex1}, a random graph is a random variable with a certain probability distribution (as opposed to a fixed graph) that depends on the values of $n$ and $p$. Assuming $n$ is given, the structure of $G_{n,p}$ changes as the value of $p$ changes, and in order to understand $G_{n,p}$, we ask questions about the structure of $G_{n,p}$ such as
\[\text{What's the probability that $G_{n,p}$ is connected?}\]
or
\[\text{What's the probability that $G_{n,p}$ is planar?}\]
Basically, for \textit{any} property $\cF (=\cF_n)$ of interest, we can ask
\[\text{What's the probability that $G_{n,p}$ satisfies property $\cF$?}\]
\nin In those questions, usually we are interested in understanding the \textit{typical} structure/behavior of $G_{n,p}$. Observe that, unless $p=0$ or $1$, there is always a positive probability that all of the edges in $G_{n,p}$ are absent, or all of them are present (see Examples \ref{ex2}, \ref{ex3}). But in this article, we would rather ignore such extreme events that happen with a tiny probability, and focus on properties that $G_{n,p}$ possesses with a probability close to 1.

We often use languages and tools from probability theory to describe/understand behaviors of $G_{n,p}$. Below we discuss some very basic examples.

 We will write $f(n) \ll g(n)$ if $\frac{f(n)}{g(n)} \ra 0$ as $n \ra \infty$.

\begin{example}\label{ex2} 

One important object in graph theory is the \textit{complete graph}, a graph with all the potential edges present. The complete graph on $n$ vertices is denoted by $K_n$.

\begin{figure}[h!]\centering \includegraphics[height=.7in]{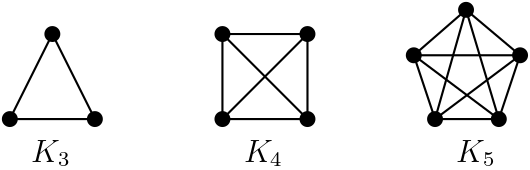}
\end{figure}

\nin We can easily imagine that, unless $p$ is very close to $1$, it is extremely unlikely that $G_{n,p}$ is complete. Indeed,
\[\pr(G_{n,p}=K_n)=p^{n \choose 2}\]
(since we want all the edges present), which tends to 0 unless $1-p$ is of order at most $n^{-2}$.

\end{example}

\begin{example}\label{ex3}
Similarly, we can compute the probability that $G_{n,p}$ is "empty" (let's denote this by $\emptyset$) meaning that no edges are present.\footnote{By the definition, $G_{n,p}$ has $n$ vertices as a default.} The probability for this event is
\[\pr(G_{n,p}=\emptyset)=(1-p)^{n \choose 2}.\] 
When $p$ is small, $1-p$ is approximately $e^{-p}$, so the above computation tells us that
\[\pr(G_{n,p}=\emptyset)\ra\begin{cases}
0 & \text{ if } p\gg 1/n^2; \\
1 & \text{ if } p\ll 1/n^2.
\end{cases}\]

\end{example}

\begin{example}\label{ex.edges}
How many edges does $G_{n,p}$ typically have? The natural first step to answer this question is computing the \textit{expected} number of edges in $G_{n,p}$. Using \textit{linearity of expectation},
    \[\begin{split}
       & \mathbb E[\text{number of edges in $G_{n,p}$}] \\
&=\sum_{i < j} \pr(\text{edge $\{i,j\}$ is present in $G_{n,p}$})\\
&= (\text{number of edges in $K_n$}) \times \pr(\text{each edge is present})\\
        &={n \choose 2} p.
    \end{split}\]
\end{example}

\begin{remark} For example, if $p=1/n$, then the expected number of edges in $G_{n,p}$ is $\frac{n-1}{2}$. But does this really imply that $G_{n,1/n}$ \emph{typically} has about $\frac{n-1}{2}$ edges? The answer to this question is related to the fascinating topic of "concentration of a probability measure." We will very briefly discuss this topic in Example \ref{ex.triangle}.
\end{remark}

\begin{example}\label{ex.triangles}

 Similarly, we can compute the expected number of \textit{triangles} (the complete graph $K_3$) in $G_{n,p}$.
    \[\begin{split}
      &  \mathbb E[\text{number of triangles in $G_{n,p}$}] \\
& =\sum_{i<j<k} \pr(\text{triangle $\{i,j,k\}$ is present in $G_{n,p}$})\\
&= (\text{number of triangles in $K_n$}) \times \pr(\text{each triangle is present})\\
        &={n \choose 3} p^3.
    \end{split}\]

\nin The above computation tells us that
\[\mathbb E[\text{number of triangles in $G_{n,p}$}]\ra\begin{cases}
0 & \text{ if } p\ll 1/n; \\
\infty & \text{ if } p\gg 1/n,
\end{cases}\]
from which we can conclude that $G_{n,p}$ is typically triangle-free if $p \ll 1/n$. (If the expectation tends to 0, then there is little chance that $G_{n,p}$ contains a triangle.)
\end{example}

\begin{remark}\label{rmk.1st} On the contrary, we \textit{cannot} conclude that $G_{n,p}$ \emph{typically} contains many triangles for $p \gg 1/n$ from the above expectation computation. Just think about a lottery of the prize money $10^{1000}$ dollars with the chance of winning $10^{-100}$, to see that a large expectation does not necessarily imply a high chance of the occurence of an event. In general, showing that a desired structure typically \textit{exists} in $G_{n,p}$ is a very challenging task, and this became a motivation for the \textit{Kahn--Kalai Conjecture} that we will discuss in the latter sections.
\end{remark}

\section{Threshold phenomena}\label{sec.TP}

One striking thing about $G_{n,p}$ is that appearance and disappearance of certain properties are abrupt. Probably one of the most well-known examples that exhibit threshold phenomena of $G_{n,p}$ is appearance of the \textit{giant component.} A \textit{component} of a graph is a maximal connected subgraph. For example, the graph in Figure \ref{fig.component} consists of four components, and the size (the number of vertices) of each component is $1, 2, 6,$ and $8$.

\begin{figure}[h!]\centering \includegraphics[height=1.2in]{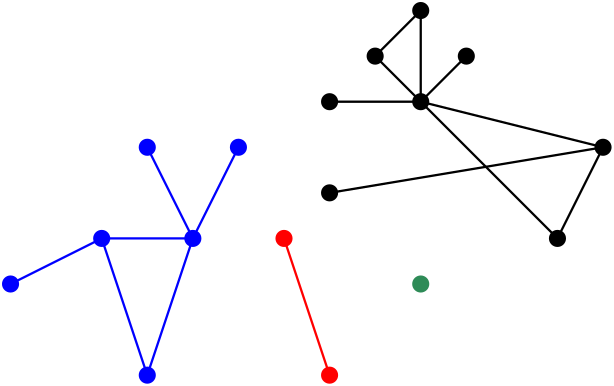}\caption{A graph that consists of four components}\label{fig.component}\end{figure}

For $G_{n,p}$, observe that, when $p=0$, the size of a largest component of $G_{n,p}$ is 1; in this case all of the edges are absent with probability 1, so each of the components is an isolated vertex. On the other hand, when $p=1$, $G_{n,p}$ is the complete graph with probability 1, so the size of its largest component is $n$.

\begin{figure}[h!]\centering\includegraphics[height=.8in]{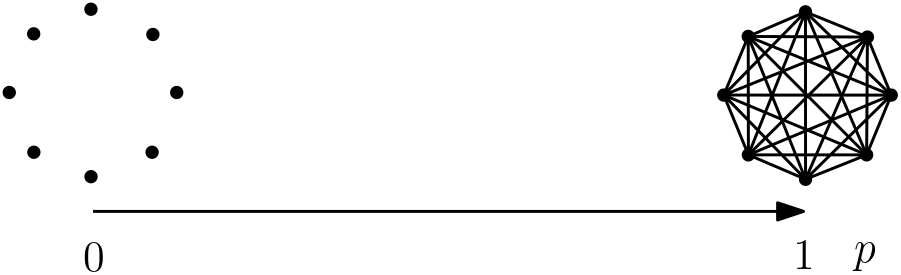}\caption{$G_{n,0}$ and $G_{n,1}$}\end{figure}

\nin Then what if $p$ is strictly between $0$ and $1$?

\begin{question}\label{Q1} What's the (typical) size of a largest component in $G_{n,p}$?\end{question}

Of course, one would naturally guess that as $p$ increases from $0$ to $1$, the typical size of a largest component in $G_{n,p}$ would also increase from 1 to $n$. But what is really interesting here is that there is a "sudden jump" in this increment.

In the following statement and everywhere else, \textit{with high probability} means that the probability that the event under consideration occurs tends to 1 as $n \ra \infty$.

\begin{theorem}[Erd\H{o}s-R\'enyi \cite{ER2}] For any $\eps>0$, the size of a largest component of $G_{n,p}$ is
\[
\begin{cases}
\le C_1(\eps) \log n & \mbox{if $np <1-\eps$} \\
\ge C_2(\eps) n & \mbox{if $np >1+\eps$}
\end{cases}\]
with high probability, where $C_1(\eps), C_2(\eps)$  depend only on $\eps$.
\end{theorem}

\nin The above theorem says that if $p$ is "slightly smaller" than $\frac{1}{n}$, then typically all of the components of $G_{n,p}$ are very small (note that $\log n$ is much smaller than the number of vertices, $n$). 

\begin{figure}[h]\centering\includegraphics[height=1.8in]{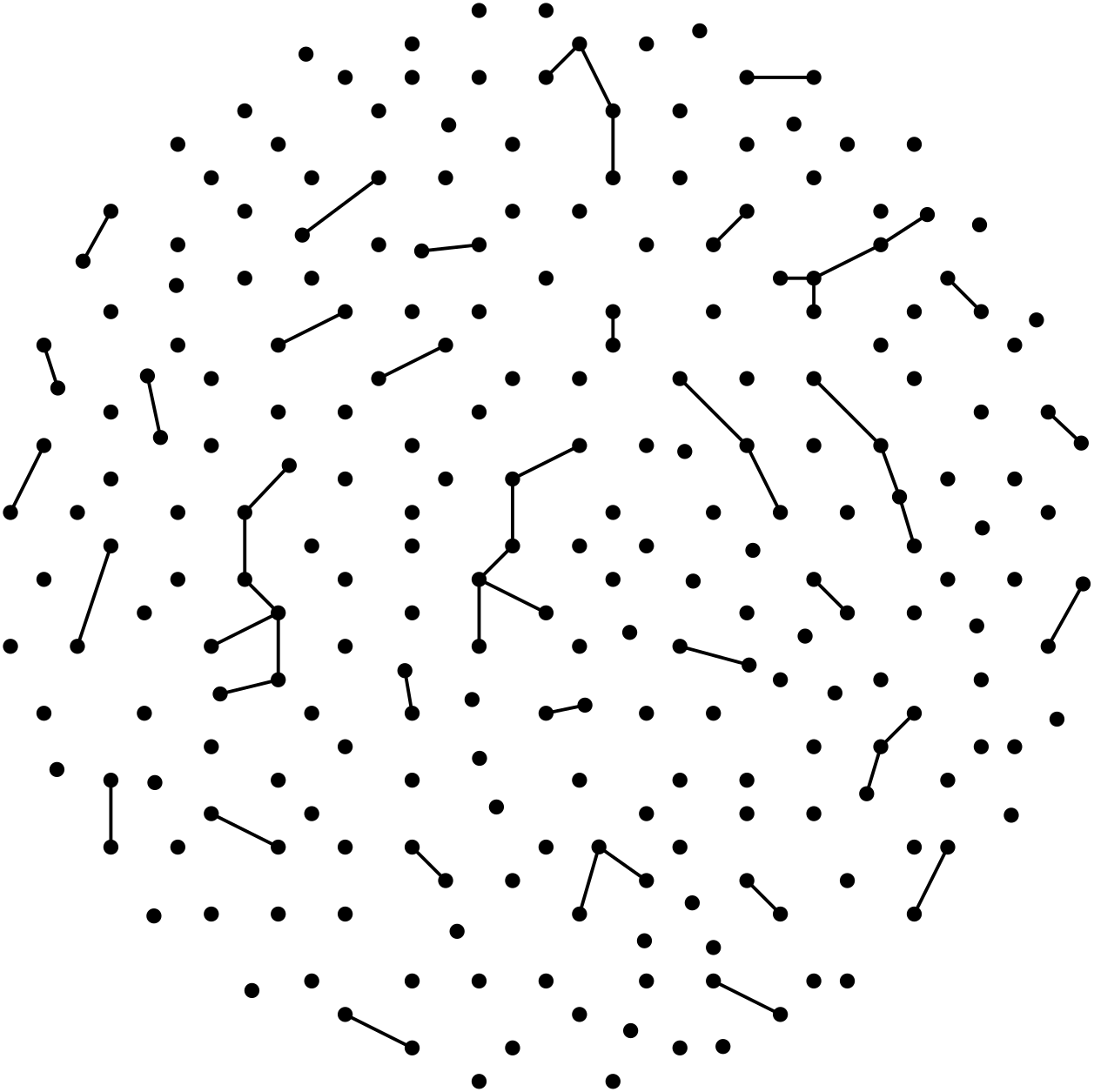}\caption{$G_{n,p}$ with all components small ($np<1-\eps$)}\end{figure}

\nin On the other hand, if $p$ is "slightly larger" than $\frac1n$, then the size of a largest component of $G_{n,p}$ is as large as linear in $n$. It is also well-known that all other components are very small (at most of order $\log n$), and this unique largest component is called the \textit{giant component}.

\begin{figure}[h]\centering\includegraphics[height=1.8in]{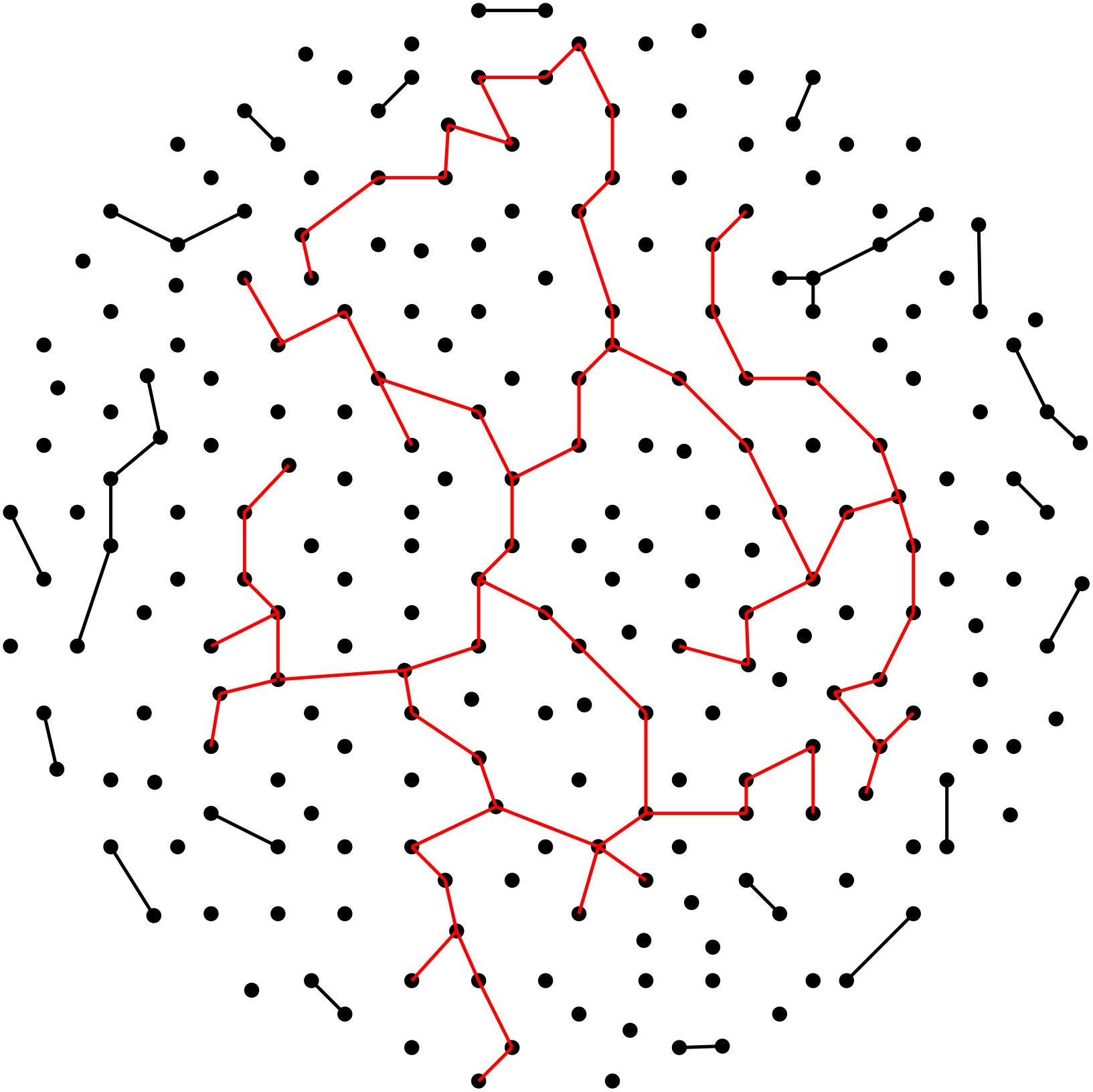}\caption{$G_{n,p}$ with the giant component ($np>1-\eps$)}\end{figure}

So around the value $p=\frac{1}{n}$, the giant component "suddenly" appears, and therefore the structure of $G_{n,p}$ also drasitically changes. This is one example of the \textit{threshold phenomena} that $G_{n,p}$ exhibits, and the value $p=\frac{1}{n}$ is a \textit{threshold function} for $G_{n,p}$ of having the giant component. (The formal definition of a threshold function is given in Definition~\ref{def.threshold}. See also the definition of \textit{the threshold} in Section \ref{sec.ETT}.)

The abrupt appearance of the giant component of $G_{n,p}$ is just one instance of vast threshold phenomena for \textit{random discrete structures}. In this article, we will mostly deal with $G_{n,p}$ for the sake of concreteness, but there will be a brief discussion about a more general setting in Section \ref{sec.ETT}.

Now we introduce the formal definition of a threshold function due to Erd\H{o}s and R\'enyi. Recall that, in $G_{n,p}$, all the vertices are labelled $1,\ldots,n$. A \textit{graph property} is a property that is invariant under graph isomorphisms (i.e. relabelling the vertices), such as $\{\text{connected}\}$, $\{\text{planar}\}$, $\{\text{triangle-free}\}$, etc. We use $\cF (=\cF_n)$ for a graph property, and $G_{n,p} \in \cF$ denotes that  $G_{n,p}$ has property $\cF$.

\begin{mydef}\label{def.threshold}
Given a graph property $\cF (=\cF_n)$, we say that $p_0=p_0(n)$ is a \emph{threshold function}\footnote{By the definition, a threshold function is determined up to a constant factor thus not unique, but conventionally people also call this \textit{the} threshold function. In this article, we will separately define \textit{the threshold} in Section \ref{sec.ETT}, which is distinguished from a threshold function.} (or simply a \emph{threshold}) for $\cF$ if
\[\pr(G_{n,p} \in \cF) \ra \begin{cases}
0 & \text{ if } p \ll p_0 \\
1 & \text{ if } p \gg p_0.
\end{cases}\]
\end{mydef}

\nin For example, $p_0=\frac{1}{n}$ is a threshold function for the existence of the giant component.

Note that it is not obvious at all whether a given graph property would admit a threshold function. Erd\H{o}s and R\'enyi proved that many graph properties have a threshold function, and about 20 years later, Bollob\'as and Thomason proved that, in fact, there is a wide class of properties that admit a threshold function. In what follows, an \textit{increasing (graph) property} is a property that is preserved under addition of edges. For example, connectivity is an increasing property, because if a graph is connected then it remains connected no matter what edges are additionally added.

\begin{theorem}[Bollobas-Thomason \cite{BT}]\label{thm.BT}
Every increasing property has a threshold function.
\end{theorem}

\nin Now it immediately follows from the above theorem that all the properties that we have mentioned so far -- connectivity, planarity\footnote{We can apply the theorem for non-planarity, which is an increasing property.}, having the giant component, etc. -- have a threshold function (thus exhibit a threshold phenomenon). How fascinating it is!

On the other hand, knowing that a property $\cF$ has a threshold function $p_0=p_0(\cF)$ does not tell us anything about the value of $p_0$. So it naturally became a central interest in the study of random graphs to find a threshold function for various increasing properties. One of the most studied classes of increasing properties is \textit{subgraph containment}, i.e., the question of for what $p=p(n)$, $G_{n,p}$ is likely/unlikely to contain a copy of the given graph. Figure \ref{well-known}  shows some of the well-known threshold functions for various subgraph containments (and that for connectivity).

\begin{figure}[h]\centering\includegraphics[height=.65in]{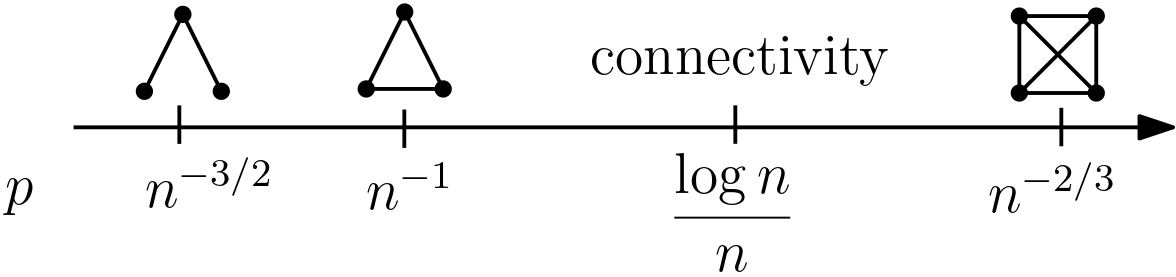}\caption{Some well-known thresholds}\label{well-known}\end{figure}

\begin{example}\label{ex.triangle}
Figure \ref{well-known} says that $p=\frac1n$ is a threshold function for the property $\cF=\{\text{contain a triangle}\}$. Recall from the definition of a threshold that this means 

\begin{enumerate}[(i)]\label{ex2.5}
\item if $ p \ll \frac1n$ then $\pr(G_{n,p} \text{ contains a triangle})\ra 0$; and
\item if $p \gg \frac1n$ then $\pr(G_{n,p} \text{ contains a triangle})\ra 1.$
\end{enumerate}

\nin We have already justified (i) in Example \ref{ex2} by showing that
\[\mathbb E[\text{number of triangles in $G_{n,p}$}] \ra 0 \text{ if } p \ll \frac{1}{n}.\]

However, showing (ii) is an entirely different story. As discussed in Remark \ref{rmk.1st}, the fact that 
\[\mathbb E[\text{number of triangles in $G_{n,p}$}] \ra \infty\]
does not necessarily imply that $G_{n,p}$ typically contains many triangles. Here we briefly describe one technique, which is called the \textit{second moment method}, that we can use to show (ii): let $X$ be the number of triangles in $G_{n,p}$, noting that  then $X$ is a random variable. By showing that the variance of $X$ is very small, which implies that $X$ is "concentrated around" $\mathbb EX$, we can derive (from the fact that $\mathbb EX$ is huge) that typically the number of triangles in $G_{n,p}$ is huge. We remark that the second moment method is only a tiny part of the much broader topic of \textit{concentration of a probability measure.}
\end{example}

We stress that, in general, finding a threshold function for a given increasing property is a very hard task. To illustrate this point, let's consider one of the most basic objects in graph theory, a \textit{spanning tree} -- a tree that contains all of the vertices. 

\begin{figure}[h]\centering\includegraphics[height=.8in]{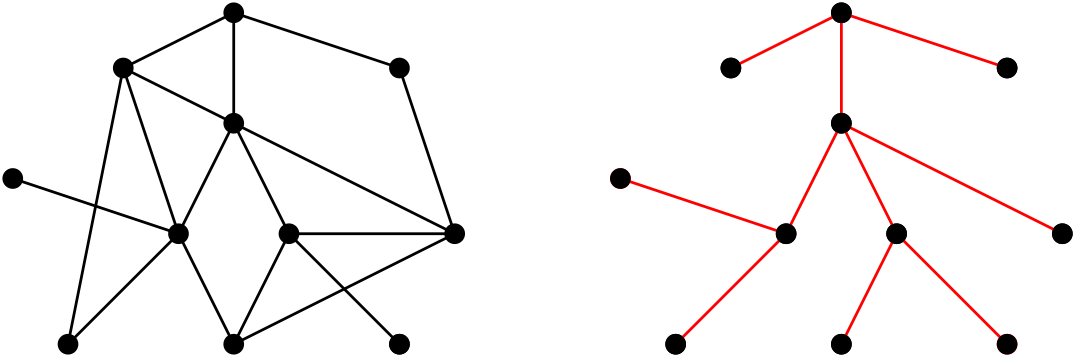}\caption{A connected graph and a spanning tree in it}\label{spanningtree}\end{figure}

The question of finding a threshold function for $G_{n,p}$ of containing a spanning tree\footnote{This is equivalent to $G_{n,p}$ is connected.} was one of the first questions studied by Erd\H{o}s and R\'enyi. Already in their seminal paper~\cite{ER2}, Erd\H{o}s and R\'enyi showed that a threshold function for containing a spanning tree is $p_0=\frac{\log n}{n}$. However, the difficulty of this problem immensely changes if we require $G_{n,p}$ to contain a \textit{specific} (up to isomorphisms) spanning tree (or more broadly, a spanning graph\footnote{A spanning graph means a graph that contains all of the vertices}). For example, one of the biggest open questions in this area back in 1960s was finding a threshold function for a \textit{Hamiltonian cycle} (a cycle that contains all of the vertices).

\begin{figure}[h]\centering\includegraphics[height=.9in]{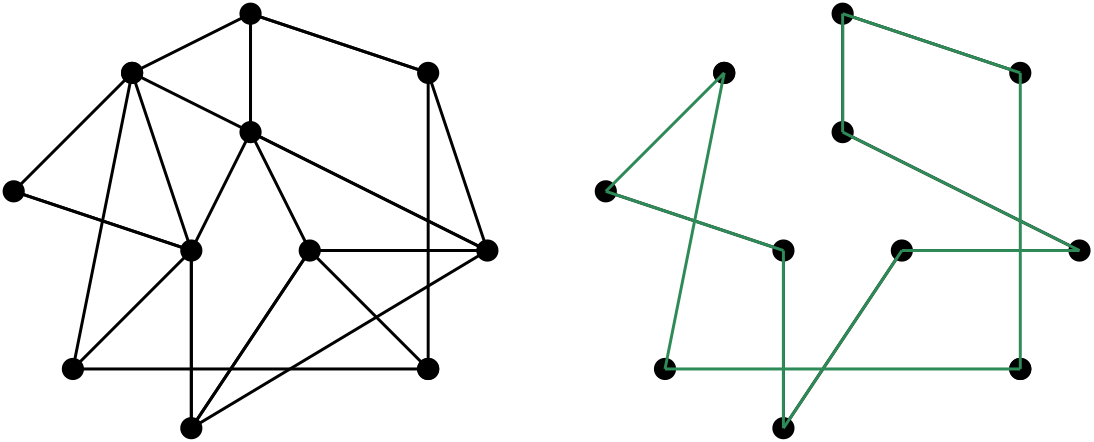}\caption{A graph and a Hamiltonian cycle in it}\label{spanningtree}\end{figure}

\nin This problem was famously solved by P\'osa in 1976.

\begin{theorem}[P\'osa \cite{Posa}]\label{thm.Posa} A threshold function for $G_{n,p}$ to contain a Hamiltonian cycle is
\[p_0(n)=\frac{\log n}{n}.\]
\end{theorem}

Note that both threshold functions for $\{$contain any spanning tree$\}$ and $\{\text{contain a Hamiltonian cycle}\}$ are of order $\frac{\log n}{n}$, even though the latter is a stronger requirement. Later we will see (in the discussion that follows Example \ref{ex.PM}) that $\frac{\log n}{n}$ is actually an easy lower bound on both threshold functions. It has long been conjectured that for \textit{any} spanning tree\footnote{More precisely, for any \textit{sequence} of spanning trees $\{T_n\}$} with a constant maximum degree, its threshold function is of order  $\frac{\log n}{n}$. This conjecture was only very recently proved by Montgomery \cite{Montgomery}.

\section{The Kahn--Kalai Conjecture: a preview}

Henceforth, $\cF$ always denotes an increasing property.

In 2006, Jeff Kahn and Gil Kalai \cite{KK} posed an extremely bold conjecture that captures the location of threshold functions for \textit{any} increasing properties. Its formal statement will be given in Conjecture~\ref{Conj.KKC.gr} (graph version) and Theorem~\ref{thm.PP} (abstract version), and in this section we will give an informal description of this conjecture first. All of the terms not defined here will be discussed in the forthcoming sections.

Given an $\cF$, we are interested in locating its threshold function, $p_0(\cF)$.\footnote{We switch the notation from $p_0(n)$ to $p_0(\cF)$ to emphasize its dependence on $\cF$.} But again, this is in general a very hard task.

Kahn and Kalai introduced another quantity which they named the \textit{expectation threshold} and denoted by $\pE(\cF)$, which is associated with some sort of expectation calculations as its name indicates. By its definition (Definition \ref{def.exp.thr}),
\[\mbox{$\pE(\cF) \le p_0(\cF)$ for any $\cF$,}\]
and, in particular, $\pE(\cF)$ is easy to compute for many interesting increasing properties $\cF$. So $\pE(\cF)$ provides an "easy" lower bound on the hard parameter $p_0(\cF)$. A really fascinating part is that then Kahn and Kalai conjectured that $p_0(\cF)$ is, in fact, bounded \textit{above} by $\pE(\cF)$ multiplied by some tiny quantity!

\begin{figure}[h!]\centering\includegraphics[height=.7in]{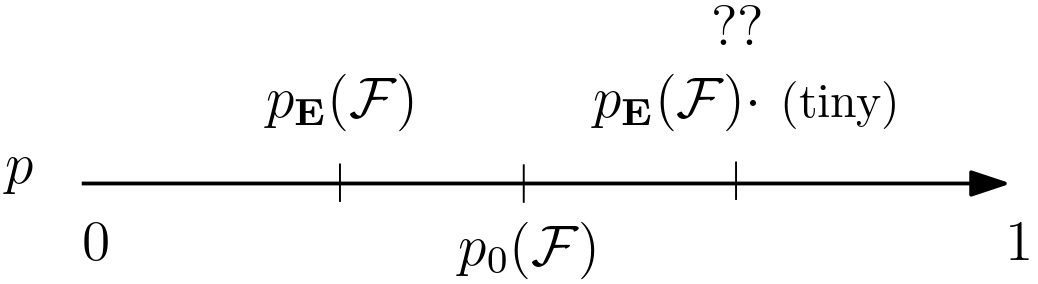}\end{figure}

\nin So this conjecture asserts that, for any $\cF$, $p_0(\cF)$ is actually well-predicted by (much) easier $\pE(\cF)$!

The graph version of this conjecture (Conjecture \ref{Conj.KKC.gr}) is still open, but the abstract version (Theorem \ref{thm.PP}) is recently proved in \cite{PPT}.

\section{Motivating examples}\label{sec.examples}

The conjecture of Kahn and Kalai is very strong, and even the authors of the conjecture wrote in their paper \cite{KK} that "it would probably be more sensible to conjecture that it is \textit{not} true." The fundamental question that motivated this conjecture was:

\begin{question}\label{Q}
What drives thresholds?
\end{question}

\nin All of the examples in this section are carefully chosen to show the motivation behind the conjecture.

Recall that the definition of a threshold (Definition \ref{def.threshold}) doesn't distinguish constant factors. So in this section, we will use the convenient notation $\gtrsim, \lesssim$, and $\asymp$ to mean (respectively) $\ge, \le$, and $=$ up to constant factors. Finally, write $p_0(H)$ for a threshold function for $G_{n,p}$ of containing a copy of $H$, for notational simplicity.

\begin{example}\label{ex.H} Let $H$ be the graph in Figure~\ref{example.H}. Let's find $p_0(H)$.

\begin{figure}[h]\centering\includegraphics[height=.5in]{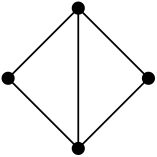}\caption{Graph $H$}\label{example.H}\end{figure}

\nin In Example \ref{ex.triangle}, we observed that there is a connection between a threshold function and computing expectations. As we did in Examples \ref{ex.edges} and \ref{ex.triangles},
    \[\begin{split}
       & \mathbb E[\text{number of $H$'s in $G_{n,p}$}] \\
&= (\text{number of (labelled) $H$'s in $K_n$}) \times \\
& \quad  \pr(\text{each (labelled) copy of $H$ is present in $G_{n,p}$})\\
        &\stackrel{(\dagger)}{\asymp} n^4 p^5,
    \end{split}\]
where $(\dagger)$ is because the number of $H$'s in $K_n$ is of order $ n^4$ (since $H$ has four vertices), and $\pr(\text{each copy of $H$ is present})$ is precisely $p^5$ (since $H$ has five edges).    So we have
    \beq{eq.H exp}
    \mathbb E[\text{number of $H$'s in $G_{n,p}$}] \ra
    \begin{cases}
        0 \quad & \text{ if } p \ll n^{-4/5};\\
        \infty \quad & \text{ if } p \gg n^{-4/5},
    \end{cases}
    \enq
and let's (informally) call the value $p=n^{-4/5}$
\[\text{"the threshold for the expectation of $H$."}\]
This name makes sense since $p=n^{-4/5}$ is where the expected number of $H$'s drastically changes. Note that \eqref{eq.H exp} tells us that
\[\pr(G_{n,p} \supseteq H) \ra 0 \quad \text{ if } p \ll n^{-4/5},\]
so, by the definition of a threshold, we have 
\[n^{-4/5} \lesssim p_0(H).\]
This way, we can always easily find a lower bound on $p_0(F)$ for any graph $F$.

What is interesting here is that, for $H$ in Figure \ref{example.H}, we can actually show that
\[\pr(G_{n,p} \supseteq H) \ra 1 \quad \text{ if } p \gg n^{-4/5}\]
using the second moment method (discussed in Example~\ref{ex.triangle}). This tells us a rather surprising fact that $p_0(H)$ is actually \textit{equal} to the threshold for the expectation of $H$.

\medskip

\nin \fbox{\begin{minipage}{9.1cm}\textbf{Dream.} Maybe $p_0(F)$ is always equal to the threshold for the expectation of $F$ for \textit{any} graph $F$?\end{minipage}}

\smallskip

\end{example}

The next example shows that the above dream is too dreamy to be true.

\begin{example}\label{example.tilde H} Consider $\tilde H$ in Figure \ref{example.H tilde} this time. Notice that $\tilde H$ is $H$ in Figure \ref{example.H} with a "tail."

\begin{figure}[h!]\centering\includegraphics[height=.5in]{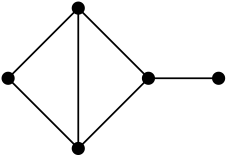}\caption{Graph $\tilde H$}\label{example.H tilde}\end{figure}

By repeating a similar computation as before, we have
    \[
    \mathbb E[\text{number of $\tilde H$'s in $G_{n,p}$}] \ra
    \begin{cases}
        0 \quad & \text{ if } p \ll n^{-5/6};\\
        \infty \quad & \text{ if } p \gg n^{-5/6},
    \end{cases}
    \]
so the threshold for the expectation of $\tilde H$ is $n^{-5/6}$. Again, this gives that
\[\pr(G_{n,p} \supseteq \tilde H) \ra 0 \quad \text{ if } p \ll n^{-5/6},\]
so we have $n^{-5/6} \lesssim p_0(\tilde H)$. However, the actual threshold $p_0(\tilde H)$ is $n^{-4/5}$, which is much larger than the lower bound.

\begin{figure}[h!]\centering\includegraphics[height=.85in]{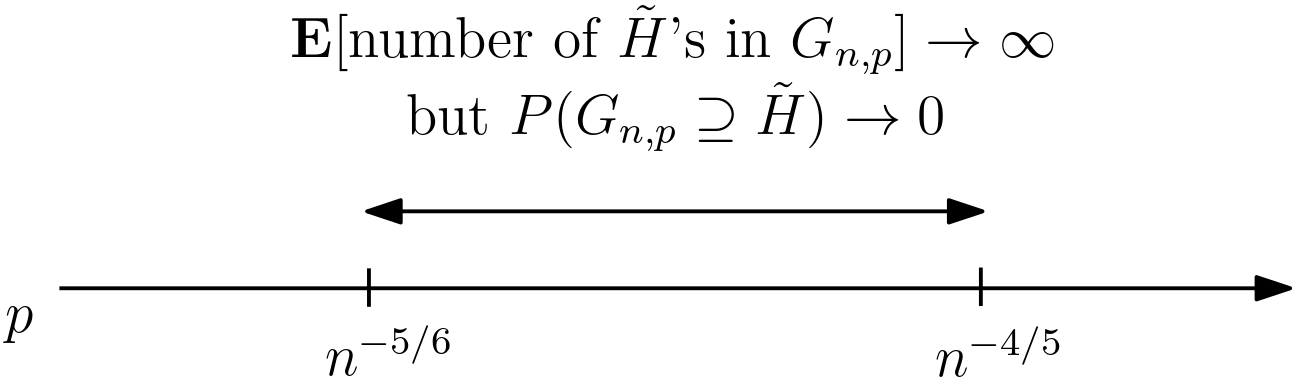}\caption{Gap between $p_0(\tilde H)$ and the expectational lower bound}\label{example.H tilde exist}\end{figure}

This is interesting, because Figure \ref{example.H tilde exist}  tells us that when $n^{-5/6} \ll p \ll n^{-4/5}$, $G_{n,p}$ contains a huge number of $\tilde H$ "on average," but still it is very unlikely that $G_{n,p}$ actually contains $\tilde H$. What happens in this inverval?

Here is an explanation. Recall from Example \ref{ex.H} that if $p \ll n^{-4/5}$, then $G_{n,p}$ is unlikely to contain $H$. But
\[\text{the absence of $H$ implies the absence of $\tilde H$,}\]
because $H$ is a subgraph of $\tilde H$! 

So when $n^{-5/6} \ll p \ll n^{-4/5}$, it is highly unlikely that $G_{n,p}$ contains $\tilde H$ because it is already unlikely that $G_{n,p}$ contains $H$. However, if $G_{n,p}$ happens to contain $H$, then that copy of $H$ typically has lots of "tails" as in Figure \ref{example.H tails}. This produces a huge number of copies of $\tilde H$'s in $G_{n,p}$.

\begin{figure}[h]\centering\includegraphics[height=.5in]{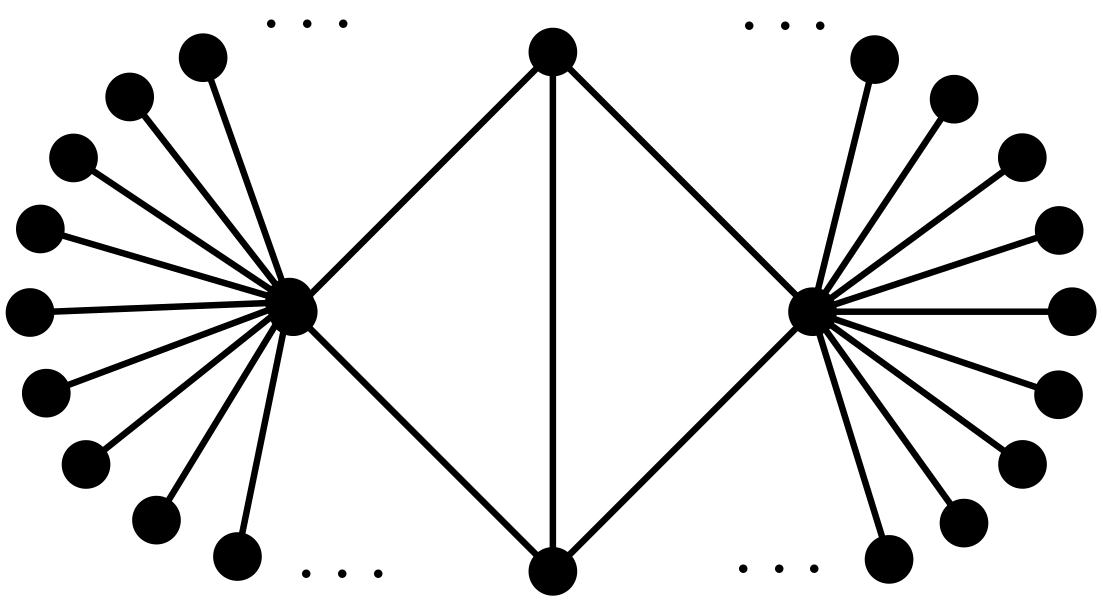}\caption{$H$ with many "tails"}\label{example.H tails}\end{figure}

\nin Maybe you have noticed the similarity between this example and the example of a lottery in Remark \ref{rmk.1st}.

\end{example}

In Example \ref{example.tilde H}, $p_0(\tilde H)$ is not predicted by the expected number of $\tilde H$, thus the \textbf{Dream} is broken. However, it still shows that $p_0(\tilde H)$ is predicted by the expected number of \textit{some} subgraph of $\tilde H$, and, intriguingly, this holds true in general. To provide its formal statement, define the \textit{density} of a graph $F$ by
\[\text{density($F$)}=\frac{\text{(the number of edges of $F$)}}{\text{(the number of vertices of $F$})}.\]
The next theorem tells us the exciting fact that we can find $p_0(F)$ by just looking at its densest subgraph, as long as $F$ is fixed.\footnote{For example, a Hamiltonian cycle is not a fixed graph, since it grows as $n$ grows.}

\begin{theorem}[Bollob\'as \cite{Bollobas}]\label{thm.bol} For any \emph{fixed} graph $F$, $p_0(F)$ is equal to the threshold for the expectation of the \emph{densest} subgraph of $F$.
\end{theorem}

For example, in Example \ref{ex.H}, the densest subgraph of $H$ is $H$ itself, so $p_0(H)$ is determined by the expectation of  $H$. This also determines $p_0(\tilde H)$ in Example \ref{example.tilde H}, since the densest subgraph of $\tilde H$ is again $H$.

Motivated by the preceding examples and Theorem \ref{thm.bol}, we give a formal definition of the \textit{expectation threshold}.

\begin{mydef} [Expectation threshold]\label{def.exp.thr} For any graph $F$, the expectation threshold for $F$ is
\[\pE(F)=\min\{p: \mathbb E[\text{number of $F'$ in $G_{n,p}$} ] \ge 1 \quad \forall F' \sub F\}.\]
\end{mydef}

Observe that
\beq{pe.lb} \pE(F) \lesssim p_0(F) \text{ for any $F$},\enq
and in particular, Theorem \ref{thm.bol} gives that
\[\pE(F) \asymp p_0(F) \text{ for any fixed $F$}.\]

\nin Note that this gives a beautiful answer to Question \ref{Q} whenever $\cF$ is a property of containing a fixed graph.

\begin{example}\label{ex.PM} Theorem \ref{thm.bol} characterizes threshold functions for any fixed graphs. To extend our exploration, in this example we consider a graph that \textit{grows} as $n$ grows. 
We say a graph $M$ is a \textit{matching} if $M$ is a disjoint union of edges. $M$ is a \textit{perfect matching} if $M$ is a matching that contains all the vertices. Write PM for perfect matching. 

\begin{figure}[h]\centering\includegraphics[height=1in]{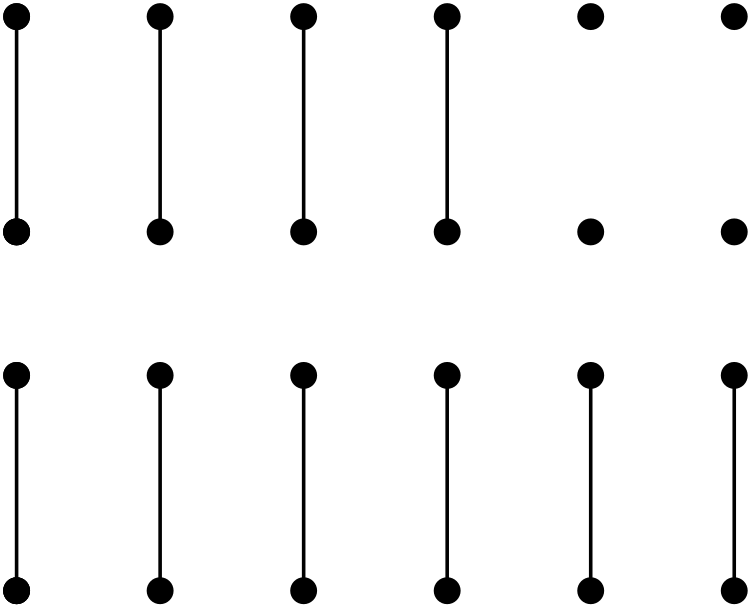}\caption{A matching (above) and a perfect matching (below)}\end{figure}

Keeping Question \ref{Q} in mind, let's first check the validity of Theorem \ref{thm.bol} to a perfect matching\footnote{We assume $2|n$ to avoid a trivial obstruction from having a perfect matching.}, which is not a fixed graph. By repeating a similar computation as before, we obtain that
\[\mathbb E[\text{number of PM's in $G_{n,p}$}] \asymp (n/e)^{n/2}p^{n/2},\]
which tends to $0$ if $p \ll 1/n.$ In fact, it is easy to compute (by considering all subgraphs of a perfect matching) that $\pE(\text{PM})\asymp 1/n$, so by \eqref{pe.lb},
\[p_0(\text{PM}) \gtrsim 1/n.\] 
However, unlike threshold functions for fixed graphs, $p_0(\text{PM})$ is \textit{not} equal to $\pE(\text{PM})$; it was proved by Erd\H{o}s and R\'enyi that 
\beq{ER.PM} p_0(\text{PM}) \asymp \frac{\log n}{n} \; (\gg \pE(\text{PM})).\enq

\begin{figure}[h!]\centering\includegraphics[height=.85in]{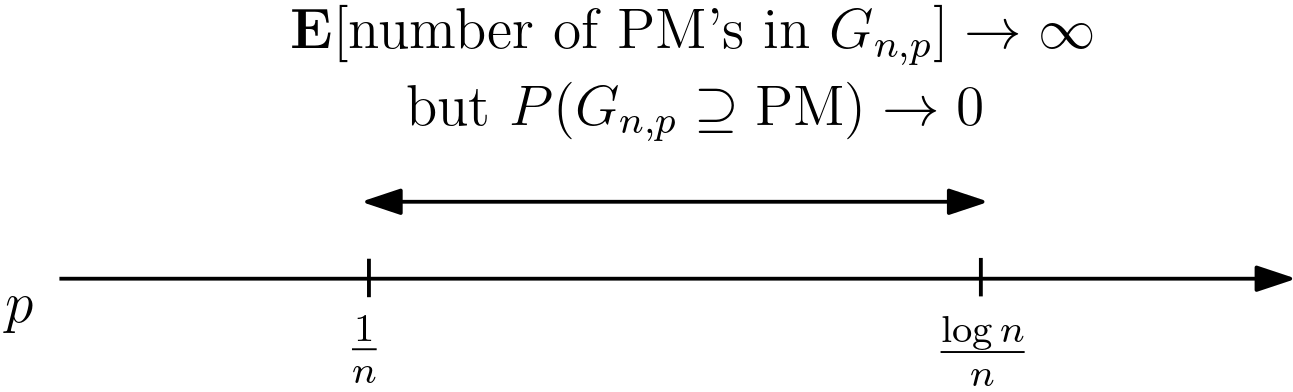}\caption{Gap between $p_0(\text{PM})$ and $\pE(\text{PM})$}\label{example.PM exist}\end{figure}

Notice that, in Figure \ref{example.PM exist}, what happens in the gap is \textit{fundamentally} different from that in Figure \ref{example.H tilde exist}. When $\frac{1}{n} \ll p \ll \frac{\log n}{n}$, $G_{n,p}$ contains huge numbers of PMs and \textit{all its subgraphs} "on average."  This means the absence of a subgraph of a PM is not the obstruction for $G_{n,p}$ from having a PM when $p\gg 1/n$. Then what happens here, and what's the real obstruction?

It turned out, we have
\[p_0(\text{PM}) \gtrsim \frac{\log n}{n}\]
for a very simple reason: the existence of an isolated vertex\footnote{a vertex not contained in any edges} in $G_{n,p}$. It is well-known that if $p \ll \frac{\log n}{n}$, then $G_{n,p}$ contains an isolated vertex with high probability (this phenomenon is elaborated in Example \ref{ex.CC}). Of course, if there is an isolated vertex in a graph, then this graph cannot contain a perfect matching.

So \eqref{ER.PM} says the very compelling fact that once $p$ is large enough that $G_{n,p}$ avoids isolated vertices, $G_{n,p}$ contains a perfect matching!
\end{example} 

The existence of an isolated vertex in $G_{n,p}$ is essentially equivalent to the \textit{Coupon-collector's Problem}:

\begin{example}[Coupon-collector]\label{ex.CC}

Each box of cereal contains a random coupon, and there are $n$ different types of coupons. If all coupons are equally likely, then how many boxes of cereal do we (typically) need to buy to collect all $n$ coupons?

The well-known answer to this question is that we need to buy $\gtrsim n\log n$ boxes of cereal. This phenomenon can be translated to $G_{n,p}$ in the following way: in $G_{n,p}$, the $n$ vertices are regarded as coupons. If a vertex is contained in a (random) edge in $G_{n,p}$, then that is regarded as being "collected." Note that if $p \ll \frac{\log n}{n}$, then typically the number of edges in $G_{n,p}$ is ${n \choose 2}p \ll n\log n$, and then the Coupon-collector's Problem says that there is typically an "uncollected coupon," which is an isolated vertex.   

\end{example}

Observe that, in Example \ref{ex.PM}, the "coupon-collector behavior" of $G_{n,p}$ provides another lower bound on $p_0(\text{PM})$, pushing up the first lower bound, $\pE(\text{PM})$, by $\log n$. And it turned out that this second (better) lower bound is equal to the threshold.

\begin{figure}[h!]\centering
\begin{tabular}{ | m{5em} | m{2.3cm}| m{2.3cm} | } 
 \hline
 \multicolumn{2}{|c|}{Lower bounds} & Threshold \\
\hline
Expectation threshold & $p_0 \gtrsim \pE$ & \multirow{2}{2.3cm}{$p_0 \asymp \pE\log n$} \\ 
 Coupon collector & $p_0 \gtrsim \pE\log n$ &\\
\hline
\end{tabular}
\caption{Bounds on $p_0(\text{PM})$}
\end{figure}

\nin \textbf{Hitting time.} Again, the existence of an isolated vertex in a graph trivially blocks this graph from containing any spanning subgraphs. In Example \ref{ex.PM}, we observed the compelling phenomenon that if $p$ is large enough that $G_{n,p}$ typically avoids  isolated vertices, then for those $p$, $G_{n,p}$ contains a perfect matching with high probability. Would this mean that, for $G_{n,p}$, isolated vertices are the \textit{only} barriers to the existence of spanning subgraphs?

To investigate this question, we consider a random process defined as below. Consider a sequence of graphs on $n$ vertices
\[G_0=\emptyset, G_1, G_2, \ldots, G_{n \choose 2}=K_n,\]
where $G_{m+1}$ is obtained from $G_m$ by adding a random edge.

\begin{figure}[h]\centering\includegraphics[height=.5in]{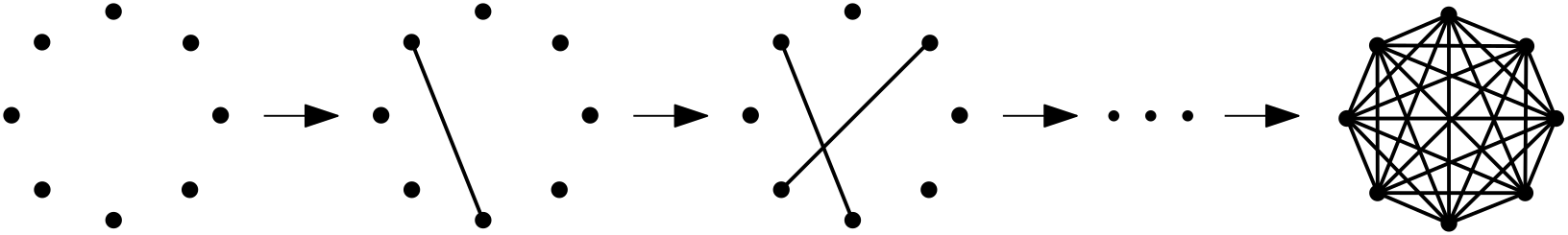}\caption{Random process}\end{figure}

\nin Then $G_m$, the $m$-th graph in this sequence, is the random variable that is uniformly distributed among all the graphs on $n$ vertices with $m$ edges. The next theorem tells us that, indeed, isolated vertices are \textit{the} obstructions for a random graph to having a perfect matching. 

\begin{theorem}[Erd\H{o}s-R\'enyi \cite{ER66}]\label{thm.hitting.PM}
Let $m_0$ denote the first time that $G_m$ contains no isolated vertices. Then, with high probability, $G_{m_0}$ contains a perfect matching.
\end{theorem}

We remark that Theorem \ref{thm.hitting.PM} gives much more precise information about $p_0(\text{PM})$ (back in $G_{n,p}$ setting). For example, Theorem \ref{thm.hitting.PM} implies:

\begin{theorem}
Let $p=\frac{\log n+c_n}{n}$. Then
\[\lim_{n \ra \infty} \pr(G_{n,p} \supseteq \text{PM} )= \begin{cases}
0 & \text{ if } c_n \ra -\infty\\
e^{-e^{-c}} & \text{ if } c_n \ra c\\
1 & \text{ if } c_n \ra \infty
\end{cases}.\]
\end{theorem}

We observe a similar phenomenon for Hamiltonian cycles. Notice that in order for a graph to contain a Hamiltonian cycle, a minimum requirement is that every vertex is contained in at least \textit{two} edges. The next theorem tells us that, again, this naive requirement is essentially the only barrier.

\begin{theorem}[Ajtai-Koml\'os-Szemer\'edi \cite{AKS}, Bollob\'as \cite{Bol84}]
Let $m_1$ denote the first time that every vertex in $G_m$ is contained in at least two edges. Then, with high probability, $G_{m_1}$ contains a Hamiltonian cycle.
\end{theorem}

Returning to Question \ref{Q}, so far we have established that there are two factors that affect threshold functions. We first observed that $\pE$ always gives a lower bound on $p_0$. We then observed that, when it applies, the coupon-collector behavior of $G_{n,p}$ pushes up this expectational lower bound by $\log n$. Conjecture \ref{Conj.KKC.gr} below dauntingly proposes that there are \textit{no other factors} that affect thresholds.

\begin{conjecture}[Kahn-Kalai \cite{KK}]\label{Conj.KKC.gr}
For any graph $F$,
\[p_0(F) \lesssim \pE(F)\log n.\]
\end{conjecture}

\nin Conjecture \ref{Conj.KKC.gr} is still wide open even after the "abstract version" of this conjecture is proved. We close this section with a very interesting example in which $p_0$ lies strictly in between $\pE$ and $\pE \log n$. A \textit{triangle factor} is a (vertex-) disjoint union of triangles that contains \textit{all} the vertices.

\begin{figure}[h]\centering\includegraphics[height=0.35in]{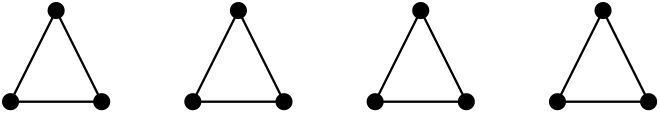}\caption{A triangle factor}\end{figure}

The question of a threshold function for a triangle-factor\footnote{or, more generally, a $K_r$-factor for any fixed $r$} was famously solved by Johansson, Kahn, and Vu \cite{JKV}. Observe that an obvious obstruction for a graph from having a triangle factor is the existence of a vertex that is not contained in any triangles. The result below is the hitting time version of \cite{JKV}, which is obtained by combining \cite{Kahn} and \cite{Heckel}.

\begin{theorem}
Let $m_2$ denote the first time that every vertex in $G_m$ is contained in at least one triangle. Then, with high probability, $G_{m_2}$ contains a triangle factor.
\end{theorem}

\nin The above theorem implies that
\[p_0(\text{triangle factor}) \asymp \pE(\text{triangle factor})\cdot (\log n)^{1/3}.\]

\section{The Expectation Threshold Theorem}\label{sec.ETT}

The abstract version of the Kahn--Kalai Conjecture, which is the main content of this section, is recently proved in \cite{PPT}. We remark that the discussion in this section is not restricted by the languages in graph theory anymore.

We introduce some more definitions for this general setting. Given a finite set $X$, the \textit{$p$-biased product probability measure}, $\mu_p$, on $2^X$ is defined by
\[\mu_p(A)=p^{|A|}(1-p)^{|X \setminus A|} \quad \forall A \sub X.\]
We use $X_p$ for the random variable whose law is
\[\pr(X_p=A)=\mu_p(A) \quad \forall A \sub X.\]
In other words, $X_p$ is a "$p$-random subset" of $X$, which means $X_p$ contains each element of $X$ with probability $p$ independently. 

\begin{example}
If $X={[n] \choose 2}$, then
\[X_p=G_{n,p}.\]
So $G_{n,p}$ is a special case of the random model $X_p$.
\end{example}

We redefine increasing property in our new set-up. A \textit{property} is a subset of $2^X$, and $\cF \sub 2^X$ is an \textit{increasing property} if 
\[B \supseteq A \in \cF \Rightarrow B \in \cF. \]
Informally, a property is increasing if we cannot "destroy" this property by adding elements. Note that in this new definition, $\cF$ is not required to possess the strong symmetry as in increasing \textit{graph} properties; for example, there is no longer a requirement like "invariant under graph isomorphisms."

Observe that $\mu_p(\cF) (:=\sum_{A \in \cF} \mu_p(A)=\pr(X_p \in \cF))$ is a polynomial in $p$, thus continuous. Furthermore, it is a well-known fact that $\mu_p(\cF)$ is strictly increasing in $p$ unless $\cF = \emptyset, 2^X$ (see Figure \ref{fig.mup}). For the rest of this section, we always assume $\cF \ne \emptyset, 2^X$.

\begin{figure}[h!]\centering\includegraphics[height=1.6in, width=1.9in]{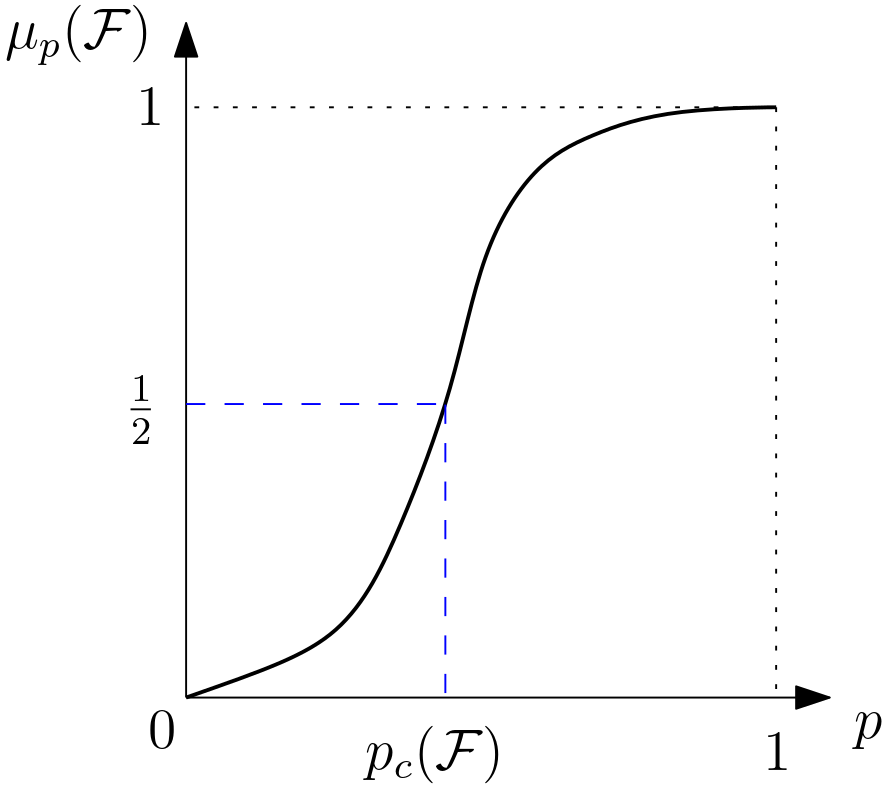}\caption{$\mu_p(\cF)$ for $p \in [0,1]$, and $p_c(\mathcal F)$}\label{fig.mup}\end{figure}

Because $\mu_p(\cF)$ is continuous and strictly increasing in $p$, there exists a unique $p_c(\cF)$ for which $\mu_{p_c}(\cF)=1/2$. This $p_c(\cF)$ is called \textit{the threshold} for $\cF$. 

\begin{remark}
The definition of $p_c(\cF)$ does not require sequences. Incidentally, by Theorem~\ref{thm.BT}, for any increasing \textit{graph} property $\cF (=\cF_n)$, $p_c(\cF)$ is an (Erd\H{o}s-R\'enyi) threshold function for $\cF$.
\end{remark}

For a general increasing property $\cF \sub 2^X$, the definition of $\pE$ is not applicable anymore. Kahn and Kalai introduced the following generalized notion of the expectation threshold, which is also introduced by Talagrand \cite{Talagrand}.

\begin{mydef}\label{def.exp.thr.abs}
Given a finite set $X$ and an increasing property $\cF\sub 2^X$, $q(\cF)$ is the maximum of $q \in [0,1]$ for which there exists $\cG \sub 2^X$ satisfying the following two properties.
\begin{enumerate}[(a)]
\item Each $A \in \cF$ contains some member of $\cG$.
\item $\sum_{S \in \cG} q^{|S|}\le 1/2$.
\end{enumerate}
A family $\cG \sub 2^X$ that satisfies (a) is called a \emph{cover} of $\cF$.
\end{mydef}

\begin{remark}
The definition of $q(\cF)$ eliminates the "symmetry" requirement -- which seems very natural (and seemingly easier to deal with) in the context of thresholds for subgraph containments -- from the definition of $\pE$. It is worth noting that this flexibility is crucially used in the proof of Theorem~\ref{thm.PP} in \cite{PPT}. 
\end{remark}

The next proposition says that $q(\cF)$ still provides a lower bound on the threshold.

\begin{prop}\label{prop.lb}
For any finite set $X$ and increasing property $\cF \sub 2^X$,
\[q(\cF) \le p_c(\cF).\]
\end{prop}

\begin{proof}[Justification.]
Write $q=q(\cF)$. By the definition of $p_c(\cF)$, it suffices to show that $\mu_q(\cF)\le 1/2$. We have
\[\begin{split}\mu_q(\cF) \le \sum_{S \in \cG} \sum_{S \sub A \in \cF} \mu_q(A) &\le \sum_{S \in \cG} \sum_{B \supseteq S} \mu_q(B) \\
& =\sum_{S \in \cG} q^{|S|} \le 1/2\end{split}\]
where the first inequality uses the fact that $\cG$ covers $\cF$.
\end{proof}

For a graph $F$, write $\cF_F$ for the increasing graph property of containing a copy of $F$. The example below illustrates the relationship between $\pE(F)$ and $q(\cF_F)$.

\begin{example}[Example \ref{example.tilde H} revisited]\label{ex'} For $X={[n] \choose 2}$ (so $X_p=G_{n,p}$) and the increasing property $\cF=\{\text{contain } \tilde H\} (\sub 2^X)$, 
\[\cG_1:=\{\text{all the (labelled) copies of $\tilde H$ in $K_n$}\}\]
is a cover of $\cF$. The left-hand side of Definition~\ref{def.exp.thr.abs} (b) is
\[\begin{split}  \sum_{S \in \cG_1} q^{|S|} = &(\text{number of $\tilde H$'s in $K_n$}) \\ & \times \pr(\text{each copy of $\tilde H$ is present in $G_{n,p}$}),\end{split}\]
which is precisely the expected number of $\tilde H$'s in $G_{n,p}$. Combined with Proposition \ref{prop.lb}, the above computation gives that $n^{-5/6} \lesssim p_c(\cF)$.

On the other hand, we have (implicitly) discussed in Example~\ref{example.tilde H} that there is another cover that gives a lower bound better than that of $\cG_1$; if we take
\[\cG_2:=\{\text{all the (labelled) copies of $H$ in $K_n$}\},\]
then the computation in Definition \ref{def.exp.thr.abs} (b) gives that $n^{-4/5} \lesssim p_c(\cF)$.

\end{example}

The above discussion shows that, for any (not necessarily fixed) graph $F$,
\[\pE(F) \lesssim q(\cF_F)\]
(whether $\pE(F) \asymp q(\cF_F)$ is unknown). The abstract version of the Kahn--Kalai Conjecture is similar to its graph version, with $\pE$ replaced by $q(\cF)$. This is what's proved in \cite{PPT}.

\begin{theorem}[Park-Pham \cite{PPT}, conjectured in \cite{KK, Talagrand}]\label{thm.PP}
    There exists a constant $K$ such that for any finite set $X$ and increasing property $\cF \sub 2^X$,
    \[p_c(\cF) \le K q(\cF)\log \ell(\cF)\]
where $\ell(\cF)$ is the size of a largest minimal element of $\cF$.
\end{theorem}

Theorem \ref{thm.PP} is extremely powerful; for instance, its immediate consequences include historically difficult results such as the resolutions of \textit{Shamir's Problem} \cite{JKV} and the \textit{"Tree Conjecture"} \cite{Montgomery}. Here we just mention one smaller consequence:

\begin{example}
If $F$ is a fixed graph, then $\ell(\cF_F)$ is the number of edges in $F$, thus a constant. So in this case Theorem \ref{thm.PP} says $p_c(\cF)\asymp q(\cF)$, which recovers Theorem~\ref{thm.bol}.
\end{example}

\nin \textbf{Sunflower Conjecture, and "fractional" Kahn-Kalai.} The proof of Theorem \ref{thm.PP} is strikingly easy given its powerful consequences. The approach is inspired by remarkable work of Alweiss, Lovett, Wu, and Zhang \cite{ALWZ} on the \textit{Erd\H{o}s-Rado Sunflower Conjecture}, which seemingly has no connection to threshold phenomena. This totally unexpected connection was first exploited by Frankston, Kahn, Nayaranan, and the author in \cite{FKNP}, where a "fractional" version of the Kahn-Kalai conjecture (conjectured by Talagrand \cite{Talagrand}) was proved, illustrating how two seemingly unrelated fields of mathematics can be nicely connected!

Note that $q(\cF)$ is in theory hard to compute. For instance, in Example \ref{example.tilde H}, we can estimate $\pE(\tilde H)$ by finding $F \sub \tilde H$ with the maximum $e(F)/v(F)$. On the other hand, to compute $q(\cF_{\tilde H})$, we should in principle consider all possible covers of $\cF_{\tilde H}$, which is typically not feasible. The good news is that there is a convenient way to find an upper bound on $q(\cF)$, which is often of the correct order. Namely, Talagrand~\cite{Talagrand} introduced a notion of \textit{fractional expectation threshold}, $q_f(\mathcal F)$, satisfying \[q(\cF) \le q_f(\cF) \le p_c(\cF)\] for any increasing property $\cF.$ He conjectured (and it was proved in \cite{FKNP}) that the (abstract) Kahn-Kalai Conjecture (now Theorem~\ref{thm.PP}) holds with $q_f(\cF)$ in place of $q(\cF)$. This puts us in linear programming territory: by LP duality, a bound $q_f(\cF) \le \alpha$ ($\alpha \in [0,1]$) is essentially equivalent to existence of an "$\alpha$-spread" probability measure on $\cF$. In all applications of Theorem \ref{thm.PP} to date, what is actually used to upper bound $q(\cF)$ is an appropriately spread measure.\footnote{The problem of constructing well-spread measure is getting growing attention now: see e.g. \cite{Kang} for a start.} So all these applications actually follow from the weaker Talagrand version.

We close this article with a very interesting conjecture of Talagrand \cite{Talagrand} that would imply the equivalence of Theorem~\ref{thm.PP} and its fractional version:

\begin{conjecture}
    There exists a constant $K$ such that for any finite set $X$ and increasing property $\cF \sub 2^X,$
    \[q_f(\mathcal F) \le Kq(\mathcal F).\]
\end{conjecture}

\nin \textit{Acknowledgement.} The author is grateful to Jeff Kahn for his helpful comments.

\end{document}